\documentclass[reqno]{amsart}

\usepackage{mystyle}

\usepackage{etoolbox}

\usepackage[headings]{fullpage}
\usepackage{parskip}

\usepackage[utf8]{inputenc}
\usepackage[T1]{fontenc}
\usepackage[british]{babel}
\usepackage[pdfencoding=auto]{hyperref}

\usepackage{microtype}
\usepackage{booktabs}

\usepackage{amsfonts,amssymb,bbm,mathrsfs,stmaryrd}
\usepackage{amsmath}

\usepackage{mathtools}
\usepackage[noabbrev,capitalize]{cleveref}
\usepackage{centernot}
\usepackage[shortlabels]{enumitem}
\usepackage{url}
\usepackage{tikz}
\usepackage{pict2e}
\usepackage{framed}

\usetikzlibrary{matrix}
\usetikzlibrary{positioning}
\usetikzlibrary{arrows}
\tikzset{> =stealth}
\usetikzlibrary{arrows.meta}
\tikzset{normalHead/.tip={Triangle[open,angle=60:4pt]},}
\tikzset{normalTail/.tip={Triangle[reversed,open,angle=60:4pt]},}

\newcommand{\addQEDstyle}[2]{\AtBeginEnvironment{#1}{\pushQED{\qed}\renewcommand{\qedsymbol}{#2}}\AtEndEnvironment{#1}{\popQED}}

\theoremstyle{plain}
\newtheorem{theorem}{Theorem}[section]
\newtheorem{lemma}[theorem]{Lemma}
\newtheorem{proposition}[theorem]{Proposition}
\newtheorem{corollary}[theorem]{Corollary}

\theoremstyle{definition}
\newtheorem{definition}[theorem]{Definition}

\newtheorem{example}[theorem]{Example}
\addQEDstyle{example}{$\triangle$}

\theoremstyle{remark}
\newtheorem{remark}[theorem]{Remark}

\renewcommand{\epsilon}{\varepsilon}
\renewcommand{\phi}{\varphi}

\newcommand{\N}{\mathbb{N}}
\newcommand{\Z}{\mathbb{Z}}

\newcommand{\inv}{^{-1}}

\newcommand{\Act}{\mathrm{Act}}

\newcommand{\WAct}{\mathrm{WAct}}

\newcommand{\Mon}{\textbf{Mon}}
\newcommand{\Ab}{\textbf{Ab}}

\newcommand{\Cat}{\mathrm{Cat}}

\newcommand{\Hom}{\mathrm{Hom}}

\newcommand{\Cext}{\mathrm{CExt}}

\DeclareMathOperator{\End}{End}

\newcommand{\splitext}[6]{%
\tikz[baseline]{
\newdimen{\mylabelwidth}
\newdimen{\skipwidth}
\node[anchor=base] (A) {\hspace*{\dimexpr0.5pt-\pgfkeysvalueof{/pgf/inner xsep}}${#1}$};
\settowidth{\mylabelwidth}{\pgfinterruptpicture {$#2$} \endpgfinterruptpicture}
\pgfmathsetlength{\skipwidth}{max(\mylabelwidth,10pt)}
;\node[right] (B) at ([xshift=\skipwidth+12pt]A.east) {${#3}$};
\settowidth{\mylabelwidth}{\pgfinterruptpicture {$#4$} \endpgfinterruptpicture}
\settowidth{\skipwidth}{\pgfinterruptpicture {$#5$} \endpgfinterruptpicture}
\pgfmathsetlength{\skipwidth}{max(\skipwidth,\mylabelwidth,10pt)}
\node[right] (C) at ([xshift=\skipwidth+12pt]B.east) {${#6}$\hspace*{\dimexpr0.5pt-\pgfkeysvalueof{/pgf/inner xsep}}};
\draw[normalTail->] (A) to node [above] {${#2}$} (B);
\draw[transform canvas={yshift=0.5ex},-normalHead] (B) to node [above] {${#4}$} (C);
\draw[transform canvas={yshift=-0.5ex},->] (C) to node [below] {${#5}$} (B);
}}

\newcommand{\normalext}[5]{%
\tikz[baseline]{
\newdimen{\mylabelwidth}
\newdimen{\skipwidth}
\node[anchor=base] (A) {\hspace*{\dimexpr0.5pt-\pgfkeysvalueof{/pgf/inner xsep}}${#1}$};
\settowidth{\mylabelwidth}{\pgfinterruptpicture {$#2$} \endpgfinterruptpicture}
\pgfmathsetlength{\skipwidth}{max(\mylabelwidth,12pt)}
\node[right] (B) at ([xshift=\skipwidth+10pt]A.east) {${#3}$};
\settowidth{\mylabelwidth}{\pgfinterruptpicture {$#4$} \endpgfinterruptpicture}
\pgfmathsetlength{\skipwidth}{max(\mylabelwidth,10pt)}
\node[right] (C) at ([xshift=\skipwidth+10pt]B.east) {${#5}$\hspace*{\dimexpr0.5pt-\pgfkeysvalueof{/pgf/inner xsep}}};
\draw[normalTail->] (A) to node [above] {${#2}$} (B);
\draw[-normalHead] (B) to node [above] {${#4}$} (C);
}}

\title{Quotients of monoid extensions and their interplay with Baer sums}

\author[P. F. Faul]{Peter F. Faul}
\address{Department of Pure Mathematics and Statistical Sciences\\ University of Cambridge}
\email{peter@faul.io}

\author[G. R. Manuell]{Graham R. Manuell}
\email{graham@manuell.me}

\date{\today}


\begin{document}

\maketitle

\begin{abstract}
 Cosetal extensions of monoids generalise extensions of groups, special Schreier extensions of monoids and Leech's normal extensions of groups by monoids. They share a number of properties with group extensions, including a notion of Baer sum when the kernel is abelian. However, unlike group extensions (with fixed kernel and cokernel) there may be nontrivial morphisms between them. We explore the structure of the category of cosetal extensions and relate it to an analogue of second cohomology groups. Finally, the order structure and additive structures are combined to give an indexed family of inverse semigroups of extensions. These in turn can be combined into an inverse category.
\end{abstract}

\setcounter{section}{-1}
\section{Introduction}

It is well understood that from each extension of groups $\normalext{N}{k}{G}{e}{H}$ with abelian kernel, a unique action $\phi$ can be extracted. These actions then partition the set of isomorphism classes of extensions. Each set in the partition can be shown to be isomorphic to the second cohomology group $\mathcal{H}^2(H,N,\phi)$ whose elements are equivalence classes of factor sets and whose operation corresponds to the Baer sum of extensions. (See \cite{maclane2012homology} for more details.)

The construction of the action above usually makes explicit use of conjugation, which presents an obstacle in generalising these ideas to the context of monoids.

Much work has gone into generalising these ideas to the context of monoids --- usually by restricting to extensions of monoids that behave sufficiently similarly to those of groups.
Such approaches include the monoid extensions of Leech \cite{leech1982extending}, the Schreier extensions in \cite{redei1952verallgemeinerung,tuen1976nonabelianextensions} and the special Schreier extensions of \cite{martins2013semidirect,martins2016baer}.
Only the case of special Schreier extensions with abelian kernel has been shown to proceed as in the classical case described above where an action can be extracted and a second cohomology group can be defined in terms of equivalence classes of factor sets.

Special Schreier extensions are those whose kernel equivalence split extensions are the so-called Schreier split extensions \cite{martins2013semidirect}, which can in turn
be thought of as the split extensions of monoids corresponding to semidirect products. 

In \cite{faul2019characterisation} a notion of a weak semidirect product was introduced and was shown to correspond precisely to the \emph{weakly Schreier split extensions} and are associated to a weak notion of action.
In \cite{faul2020baer} special Schreier extensions were generalised to \emph{cosetal extensions} --- extensions whose corresponding kernel equivalence split extension is only required to be weakly Schreier. This is equivalent to requiring that the extension $\normalext{N}{k}{G}{e}{H}$ satisfy that whenever $e(g) = e(g')$, there exist $n \in N$ such that $g = k(n)g'$.
Cosetal extensions always have group kernel and a weakly Schreier split extension is cosetal if and only if its kernel is a group.
In addition to generalising special Schreier extensions, cosetal extensions also generalise Leech's normal extensions of groups by monoids.

We have recently become aware that cosetal extensions were first studied in a rather dense paper by Fleischer \cite{fleischer1981monoid}, which does not appear to be well known. However, Fleischer does not discuss the Baer sums studied in \cite{faul2020baer}.

For cosetal extensions with abelian kernel a weak action can be extracted. Parameterising by these `actions' allows for the definition of an analogue of the second cohomology groups and a Baer sum of cosetal extensions. 
Unlike the usual actions of groups or monoids, these weak actions form a nontrivial poset. This paper is concerned with the interplay of this poset structure and the second cohomology groups.

\section{Background}
In this section we will discuss the basic theory of weakly Schreier extensions as developed in \cite{faul2019characterisation}, though it is presented here in a slightly different way.
We will also give an overview of the `second cohomology groups' for cosetal extensions with abelian kernel from \cite{faul2020baer}.
(We call these second cohomology groups by analogy to the group case, but do note that a comprehensive cohomology theory has not yet been developed in this context. This is a topic for further research.)

\subsection{Weakly Schreier split extensions}
Between any two monoids $X$ and $Y$ we have the constant $1$ homomorphism which we (somewhat confusingly) write as $0_{X,Y}$ and call the \emph{zero morphism} from $X$ to $Y$. This name is justified by the fact that if $f\colon X' \to X$ and $g\colon Y \to Y'$ we have $0_{X,Y}f = 0_{X',Y}$ and $g0_{X,Y} = 0_{X,Y'}$
and so a zero morphism behaves like $0$ with respect to composition. 

Equalisers and coequalisers exist in the category $\Mon$ of monoids. Thus, for any morphism $f\colon X \to Y$ we can ask for the equaliser of $f$ with $0_{X,Y}$ or the coequaliser of $f$ with $0_{X,Y}$; the former is called the kernel of $f$ and the latter the cokernel.

The kernel of $f$ can be shown to be the inclusion of the submonoid of elements sent by $f$ to $1$. The cokernel of $f$ is the quotient map corresponding to the congruence generated by $f(x) \sim 1$. The existence of kernels and cokernels allows us to define extensions in $\Mon$.

\begin{definition}
    A diagram $\normalext{N}{k}{G}{e}{H}$ of monoids is an \emph{extension} when $k$ is the kernel of $e$ and $e$ is the cokernel of $k$.
    
    A morphism of extensions (of $H$ by $N$) is a homomorphism $f\colon G_1 \to G_2$ making the following diagram commute.
    \begin{center}
    \begin{tikzpicture}[node distance=2.0cm, auto]
    \node (A) {$N$};
    \node (B) [right of=A] {$G_1$};
    \node (C) [right of=B] {$H$};
    \node (D) [below of=A] {$N$};
    \node (E) [right of=D] {$G_2$};
    \node (F) [right of=E] {$H$};
    \draw[normalTail->] (A) to node {$k_1$} (B);
    \draw[-normalHead] (B) to node {$e_1$} (C);
    \draw[normalTail->] (D) to node {$k_2$} (E);
    \draw[-normalHead] (E) to node {$e_2$} (F);
    \draw[->] (B) to node {$f$} (E);
    \draw[double equal sign distance] (A) to (D);
    \draw[double equal sign distance] (C) to (F);
    \end{tikzpicture}
    \end{center}
\end{definition}

A \emph{split extension} is an extension $\normalext{N}{k}{G}{e}{H}$ equipped with a monoid homomorphism $s\colon H \to G$ splitting $e$. A morphism of split extensions is a morphism of the underlying extensions commuting with the sections.

Split extensions of groups are particularly well behaved.
If $\splitext{N}{k}{G}{e}{s}{H}$ is a split extension of groups, then it is the case that for all $g \in G$, there exists a unique $n \in N$ such that $g = k(n)se(g)$. When this condition is imposed on a split extension of monoids we get the \emph{Schreier split extensions} of \cite{bourn2015Sprotomodular}. If we do away with the uniqueness requirement we obtain the weakly Schreier extensions first defined in \cite{bourn2015partialMaltsev}.

\begin{definition}
    A split extension $\splitext{N}{k}{G}{e}{s}{H}$ is \emph{weakly Schreier} if for all $g \in G$ there exists an $n \in N$ such that $g = k(n)se(g)$. 
\end{definition}

There are many naturally occurring examples of weakly Schreier split extensions --- for instance, Artin glueings of frames \cite{sga4vol1,faul2019artin} and $\lambda$-semidirect products of inverse monoids \cite{billhardt1992wreath,faul2020lambda}.

In the category of groups, split extensions between $H$ and $N$ are in one-to-one correspondence with semidirect products of $H$ and $N$ which themselves correspond to actions of $H$ on $N$. The characterisation of weakly Schreier extensions makes extensive use of the ideas behind this correspondence, so let us discuss it now in some depth.

If $\splitext{N}{k}{G}{e}{s}{H}$ is a split extension of groups then the fact that there is a unique $n$ such that $g = k(n)se(g)$ means that the function $f\colon N \times H \to G$, where $f(n,h) = k(n)s(h)$, is a bijection of sets. We would like to equip the set $N \times H$ with a group structure so that $f$ is a morphism of split extensions as in the diagram

\begin{center}
   \begin{tikzpicture}[node distance=2.0cm, auto]
    \node (A) {$N$};
    \node (B) [right of=A] {$N \times H$};
    \node (C) [right of=B] {$H$};
    \node (D) [below of=A] {$N$};
    \node (E) [right of=D] {$G$};
    \node (F) [right of=E] {$H$};
    \draw[normalTail->] (A) to node {$k'$} (B);
    \draw[transform canvas={yshift=0.5ex},-normalHead] (B) to node {$e'$} (C);
    \draw[transform canvas={yshift=-0.5ex},->] (C) to node {$s'$} (B);
    \draw[normalTail->] (D) to node {$k$} (E);
    \draw[transform canvas={yshift=0.5ex},-normalHead] (E) to node {$e$} (F);
    \draw[transform canvas={yshift=-0.5ex},->] (F) to node {$s$} (E);
    \draw[->] (B) to node {$f$} (E);
    \draw[double equal sign distance] (A) to (D);
    \draw[double equal sign distance] (C) to (F);
   \end{tikzpicture}
\end{center}
where $k'(n) = (n,1)$, $e'(n,h) = h$ and $s'(h) = (1,h)$. Since $f(n,h)f(n',h') = k(n)s(h)k(n')s(h')$, we require that $f((n,h)(n',h')) = k(n)s(h)k(n')s(h')$.
Notice that there exists a unique $x \in N$ such that $s(h)k(n) = k(x)s(h)$, namely $x = s(h)k(n)s(h)\inv$.
Denote $x$ by $\alpha(h,n)$ and observe that $f(n\alpha(h,n'),hh') = k(n)k\alpha(h,n')s(h)s(h') = k(n)s(h)k(n')s(h')$. Consequently, if we define $(n,h)(n',h') = (n\alpha(h,n'),hh')$, then $f$ becomes an isomorphism of groups and, in fact, an isomorphism of split extensions. Moreover, the map $\alpha$ is a group action of $H$ on $N$ and all actions give rise to split extensions by the above procedure.

This result generalises to Schreier split extensions of monoids, which are in correspondence with monoid actions of $H$ on $N$. (These actions play the same role as conjugation in the group case.) The situation is not as simple for weakly Schreier split extensions, though we do obtain a correspondence for a suitable weakening of action.

If $\splitext{N}{k}{G}{e}{s}{H}$ is a weakly Schreier extension, then $f\colon (n,h) \mapsto k(n)s(h)$ is no longer a bijection, but only a surjection. Notice however that if $E$ is the equivalence relation on $N \times H$ defined by ${(n,h) \sim (n',h')} \iff k(n)s(h) = k(n')s(h')$, then $f$ factors through $(N \times H)/E$ to give a bijection. As before, we would like for this bijection to be an isomorphism and so we must find a suitable multiplication for $(N \times H)/E$.

Before we do so, let us first discuss the equivalence relation $E$. There are a number of properties $E$ always satisfies, which gives rise to the following definition.

\begin{definition}
    An equivalence relation $E$ on $N \times H$ is \emph{admissible} if 
    \begin{enumerate}[start=0]
        \item $(n,h) \sim (n',h')$ implies $h = h'$,
        \item $(n,1) \sim (n',1)$ implies $n = n'$,
        \item $(n,h) \sim (n',h)$ implies $(xn,h) \sim (xn',h)$ and
        \item $(n,h) \sim (n',h)$ implies $(n,hy) \sim (n',hy)$.
    \end{enumerate}
\end{definition}

The first condition says two pairs will only be related if they agree in the second component. This allows us to reformulate an admissible equivalence relation as a particular $H$-indexed equivalence relation on $N$: each $h \in H$ gives an equivalence relation on $N$ written $n \sim^h n'$.

\begin{definition}\label{def:admissible_indexed_equivalence_relation}
    An $H$-indexed equivalence relation $E$ on $N$ is admissible if
    \begin{enumerate}
        \item $n \sim^1 n'$ implies $n = n'$,
        \item $n \sim^h n'$ implies $xn \sim^h xn'$,
        \item $n \sim^h n'$ implies $n \sim^{hy} n'$.
    \end{enumerate}
\end{definition}

For the remainder of the paper we will make use of this new perspective and consider admissible $H$-indexed equivalence relations on $N$.

For each equivalence relation $\sim^h$ we can quotient $N$ to get $N/{\sim^h}$ whose elements we write as $[n]_h$. It clear that $(N \times H)/E$ is isomorphic to $\bigsqcup_{h \in H}(N/{\sim^h})$ where $[(n,h)]$ corresponds to $([n]_h,h)$. For convenience we will later omit the subscript in the first component, as it is completely determined by the second component.

Now let us return to our weakly Schreier extension $\splitext{N}{k}{G}{e}{s}{H}$. We know that we have a corresponding admissible $H$-indexed equivalence relation $E$ and that $\bigsqcup_{h \in H}(N/{\sim^h})$ is isomorphic to $G$ via $\overline{f}([n],h) = k(n)s(h)$.

Now by the weakly Schreier condition and the axiom of choice, we know that there exists a set map $q \colon G \to N$ such that $g = kq(g)se(g)$ for all $g \in G$.
As before we can define a map $\alpha\colon H \times N \to N$ such that $\alpha(h,n) = q(s(h)k(n))$. Notice that $\overline{f}([n\alpha(h,n')],hh') = k(n)k\alpha(h,n')s(h)s(h') = k(n)s(h)k(n')s(h')$ which implies that $\overline{f}$ is an isomorphism with respect to the multiplication $([n],h)([n'],h') = ([n\alpha(h,n')],hh')$. Furthermore, it is easy to see this gives an isomorphism of extensions. 

We can characterise the functions $\alpha$ arising in this way as follows.
\begin{definition}\label{def:compatible_action}
    A map $\alpha\colon H \times N \to N$ is a \emph{compatible action} with respect to an admissible $H$-indexed equivalence relation $E$ if it satisfies the following conditions.
    \begin{enumerate}
        \item $n \sim^h n'$ implies $n\alpha(h,x) \sim^h n'\alpha(h,x)$ for all $x \in N$,
        \item $n \sim^h n'$ implies $\alpha(x,n) \sim^{xh} \alpha(x,n')$ for all $x \in H$,
        \item $\alpha(h,nn') \sim^h \alpha(h,n)\cdot \alpha(h,n')$,
        \item $\alpha(hh',n) \sim^{hh'} \alpha(h,\alpha(h',n))$,
        \item $\alpha(h,1) \sim^h 1$,
        \item $\alpha(1,n) \sim^1 n$. 
    \end{enumerate}
    We say that the $H$-indexed equivalence relation and the candidate action are \emph{compatible}
    and call $(E, \phi)$ a \emph{compatible pair}.
\end{definition}

Compatible pairs give a full characterisation of weakly Schreier extensions.
\begin{proposition}
    Let $N$ and $H$ be monoids and $(E, \phi)$ be a compatible pair. Then $\bigsqcup_{h \in H}(N/{\sim^h})$ becomes a monoid when equipped with unit $([1],1)$ and a product given by 
    \[
        ([n],h) \cdot ([n'],h') = ([n\alpha(h,n')],hh').
    \]
    We denote this monoid by $N \rtimes_{E,\phi} H$.
\end{proposition}

\begin{theorem}
    Let $N$ and $H$ be monoids and $(E,\phi)$ a compatible pair. Then the diagram \[\splitext{N}{k}{N \rtimes_{E,\phi} H}{e}{s}{H}\]
    where $k(n) = ([n],1)$, $e([n],h) = h$ and $s(h) = ([1],h)$ is a weakly Schreier extension. Furthermore, every weakly Schreier extensions is isomorphic to one of this form.
\end{theorem}

In general there can be many actions compatible with an $H$-indexed equivalence relation $E$ that give rise to the same monoid. This can be remedied by considering equivalence classes of actions where $\phi \sim \phi'$ if and only if $\phi(h,n) \sim^h \phi'(h,n)$ for all $h \in H$ and $n \in N$. This then gives a complete characterisation of weakly Schreier extensions.

A final point worth emphasising regards morphisms between weakly Schreier extensions. In the group setting we have the split short five lemma, which says that a morphism of split extensions is an isomorphism, and moreover, any such isomorphism is unique.
A similar result holds in the generalisation to Schreier split extensions of monoids. In our setting it does not. However, we show below that all morphisms between two such extensions must be unique.

Consider a morphism of split extensions.
\begin{center}
   \begin{tikzpicture}[node distance=2.0cm, auto]
    \node (A) {$N$};
    \node (B) [right=1.2cm of A] {$N \rtimes_{E,\phi} H$};
    \node (C) [right=1.2cm of B] {$H$};
    \node (D) [below of=A] {$N$};
    \node (E) [below of=B] {$N \rtimes_{E',\phi'} H$};
    \node (F) [below of=C] {$H$};
    \draw[normalTail->] (A) to node {$k_1$} (B);
    \draw[transform canvas={yshift=0.5ex},-normalHead] (B) to node {$e_1$} (C);
    \draw[transform canvas={yshift=-0.5ex},->] (C) to node {$s_1$} (B);
    \draw[normalTail->] (D) to node {$k_2$} (E);
    \draw[transform canvas={yshift=0.5ex},-normalHead] (E) to node {$e_2$} (F);
    \draw[transform canvas={yshift=-0.5ex},->] (F) to node {$s_2$} (E);
    \draw[->] (B) to node {$f$} (E);
    \draw[double equal sign distance] (A) to (D);
    \draw[double equal sign distance] (C) to (F);
   \end{tikzpicture}
\end{center}
Since $fk_1 = k_2$ and $fs_1 = s_2$, we have that $f([n],1) = ([n],1)$ and $f([1],h) = ([1],h)$. Now it is not hard to see that $([n],h) = ([n],1)([1],h)$ and so, since $f$ preserves the multiplication, we find that $f([n],h) = ([n],h)$. This unique map will only be well defined when $E \subseteq E'$ and when $\phi(h,n) \sim_{E'}^h \phi'(h,n)$. This makes the category of weakly Schreier extensions a preorder. This is equivalent to a poset $\WAct(H,N)$ whose objects are compatible pairs $(E,[\phi]_E)$. 
Here $(E,[\phi]_E) \le (E',[\phi']_{E'})$ if and only if $E \subseteq E'$ and $[\phi]_{E'} = [\phi']_{E'}$. Notice this means that we can then always rewrite such an inequality as $(E,[\phi]_E) \le (E,[\phi]_{E'})$ with the same representative action.

\subsection{Cosetal extensions and Baer sums}
We now turn our attention to general (non-split) extensions.
In all that follows we assume that the kernel $N$ is abelian group. Let us begin by briefly describing the situation in the category of groups.

For an extension of groups $\normalext{N}{k}{G}{e}{H}$ note that $e(gk(n)g\inv) = 1$ and so there exists an element $\theta(g,n)$ such that $k\theta(g,n) = gk(n)g\inv$.
This map $\theta$ is an action of $G$ on $N$. Now for any set-theoretic section $s$ of $e$ we can define $\phi \colon H \times N \to N$ by $\phi(h,n) = \theta(s(h),n)$.
It turns out that each choice of splitting gives the same map $\phi$, and furthermore, that $\phi$ is a group action of $H$ on $N$. As above, this action satisfies the important identity that $k\phi(h,n)s(h) = s(h)k(n)$.

One can show that if $e(g) = e(g')$, there is a unique $n \in N$ such that $g = k(n)g'$. Now notice that $e(s(hh')) = e(s(h)s(h'))$, so that there exists an $x \in N$ such that $k(x)s(hh') = s(h)s(h')$. Define a map $g_s \colon H \times H \to N$ that chooses these elements so that $kg_s(h,h')s(hh') = s(h)s(h')$. We call $g_s$ the \emph{associated factor set} of $s$.

Observe that $g_s(h,1) = 1 = g_s(1,h)$. Furthermore, we have that $g_s(x,y)g_s(xy,z) = \phi(x,g_s(y,z))g_s(x,yz)$. Any map $g$ satisfying the above equations is called a \emph{factor set}. It
is not hard to see that the set of factor sets is an abelian group under pointwise multiplication.

Now let $\normalext{N}{k}{G}{e}{H}$ be an extension and $s$ an arbitrary set-theoretic section of $e$. Let $\phi$ be the associated action and $g_s$ the factor set associated to $s$. We now define $N \rtimes_\phi^{g_s} H$, a generalisation of the semidirect product whose underlying set is $N \times H$ and with multiplication given by 
\[
    (n,h)(n',h') = (n\phi(h,n')g_s(h,h'),hh').
\]
The function $f\colon N \rtimes_\phi^{g_s} H \to G$ given by $f(n,h) = k(n)s(h)$ can be easily seen to be a bijection. Furthermore, the following calculation shows that it preserves the multiplication.
\begin{align*}
    f((n,h)(n',h'))     &= f(n\phi(h,n')g(h,h'),hh') \\
                        &= k(n)k\phi(h,n')kg(h,h')s(hh') \\
                        &= k(n)\phi(h,n')s(h)s(h') \\
                        &= k(n)s(h)k(n')s(h') \\
                        &= f(n,h)f(n',h')
\end{align*}

If $\phi$ is an action of $H$ on $N$ and $g$ is a factor set, then $N \rtimes_\phi^g H$ can always be made into an extension with associated action $\phi$. Thus, it seems likely that extensions can be characterised by actions and factor sets. The only hurdle is that our above calculation made use of an arbitrary splitting $s$. If instead we chose a different splitting $s'$, we would get a different factor set $g_{s'}$ which would give back the same extension. Therefore, we need to take a quotient of the abelian group of factor sets.

The appropriate subgroup to quotient by is the subgroup of \emph{inner factor sets}: factor sets of the form $\delta t(h,h') = \phi(h,t(h'))t(hh')\inv t(h)$ where $t\colon H \to N$ is any identity preserving function. The resulting quotient group is the second cohomology group $\mathcal{H}^2(H,N,\phi)$, which by the above arguments can be easily seen to correspond to the set of isomorphism classes of extensions with action $\phi$. The bijection endows this set with a natural abelian group structure.

In this paper, we are concerned with cosetal extensions of monoids with abelian kernel.
\begin{definition}
    An extension $\normalext{N}{k}{G}{e}{H}$ of monoids is (right) \emph{cosetal} if whenever $e(g) = e(g')$ there exists an $n \in N$ such that $g = k(n)g'$.
\end{definition}

These are precisely the extensions for which $H$ can be considered to be the monoid of right cosets of $N$ with the usual multiplication. A weakly Schreier split extension is cosetal if and only if it has group kernel.

If $\normalext{N}{k}{G}{e}{H}$ is a cosetal extension and $s$ is a set-theoretic section of $e$, then we can define an $H$-indexed equivalence relation $E$ by $n \sim^h n'$ if and only if $k(n)s(h) = k(n')s(h)$. This will be admissible, and moreover, any other section $s'$ will induce the same $H$-indexed equivalence relation. In this way we can associate a uniquely-defined admissible $H$-indexed equivalence relation to each cosetal extension.

Now note that $e(s(h)) = e(s(h)k(n))$ and thus by the cosetal property there exists an element $x \in N$ such that $k(x)s(h) = s(h)k(n)$. Let us denote such an element $x$ by $\phi(h,n)$.
It can be shown that $\phi\colon H \times N \to N$ is compatible action for the associated equivalence relation $E$.
Note that $\phi$ is not necessarily unique, but any other action constructed in this way will be equivalent to $\phi$ with respect to $E$.

Thus, to each cosetal extension we can associate an admissible equivalence relation $E$ and a compatible class of actions $[\phi]$.
An alternate way of viewing this is that we associate a weak semidirect product is to each cosetal extension, as one might associate a semidirect product to extensions in the group setting.

As in the group case, for each set-theoretic splitting $s$ of $e$ we can consider an associated factor set $g_s$ satisfying $kg_s(h,h')s(hh') = s(h)s(h')$. This motivates our approach to the generalisation of the second cohomology groups.

The isomorphism classes of cosetal extensions can be partitioned according to the value of $(E,[\phi])$ in each case.
Let us call each component $\Cext(H,N,E,[\phi])$.
We can characterise each extension in $\Cext(H,N,E,[\phi])$ with a relaxed notion of a factor set.
This is a function $g \colon H \times H \to N$ such that $g(h,1) \sim^h 1 \sim^h g(1,h)$ and $g(x,y)g(xy,z) \sim^{xyz} \phi(x,g(y,z))g(x,yz)$. Given $(E,[\phi])$ and a factor set $g$ we can construct a monoid $N \rtimes_{E,\phi}^g H$ with underlying set $\bigsqcup_{h \in H}N/{\sim_h}$ and multiplication given by 
\[
    ([n],h)([n'],h') = ([n\phi(h,n')g(h,h')],hh').
\]
We obtain a cosetal extension \normalext{N}{k}{N \rtimes_{E,\phi}^g H}{e}{H} where $k$ and $e$ are the obvious morphisms.

As discussed above, if we start with a cosetal extension $\normalext{N}{k}{G}{e}{H}$ and take an arbitrary section $s$, we can associate to it a factor set $g_s$ and can show that $N \rtimes_{E,\phi}^{g_s} H$ is isomorphic to $G$ and that this can further be extended to an isomorphism of extensions.

The same issue of multiple factor sets giving rise to the same extension occurs and it is again solved by quotienting by an appropriate notion of a inner factor set where we say $g \equiv g'$ if and only if there exists an inner factor set $\delta t$ such that $(\delta t + g')(h,h') \sim^{hh'} g(h,h')$ for all $h, h' \in H$. This gives rise to the second cohomology group $\mathcal{H}^2(H,N,E,[\phi])$ which is then isomorphic to $\Cext(H,N,E,[\phi])$.

\subsection{Relation to Leech and special Schreier extensions}

Let us discuss the relationship between cosetal extensions and other notions of monoid extension found in the literature.
In \cite{leech1982extending} Leech considers extensions of groups by monoids $\normalext{N}{k}{G}{e}{H}$ in which $N$ is a normal subgroup of $G$ in the sense that $N$ is a subgroup of the group of units of $G$ and $gN = Ng$ for all $g \in G$. It is easy to see that every such extension is cosetal. However, the following example demonstrates that not all cosetal extensions are Leech extensions, even when the kernel is an abelian group.

\begin{example}
 Let $\mathbbm{2}=\{\top,\bot\}$ denote the two-element meet-semilattice and consider the action $\alpha\colon \mathbbm{2} \times \Z \to \Z$ defined by $\alpha(\top,n) = n$ and $\alpha(\bot,n) = 0$. We may construct the semidirect product $\Z \rtimes_\alpha \mathbbm{2}$ and the extension $\normalext{\Z}{k}{\Z \rtimes_\alpha \mathbbm{2}}{e}{\mathbbm{2}}$ in which $k(n) = (n,\top)$ and $e(n,x) = x$.
 
 Explicitly, the multiplication in $\Z \rtimes_\alpha \mathbbm{2}$ is given by $(n,\top)\cdot (n',h) = (n+n',h)$ and $(n,\bot)\cdot (n', h) = (n, \bot)$. The right coset $\Z\cdot(0,\bot)$ is then $\Z \times \{\bot\}$, while the left coset $(0,\bot)\cdot\Z$ is $\{(0,\bot)\}$.
 Thus this extension is not a Leech extension, but it is cosetal as it is a (weakly) Schreier split extension with group kernel.
\end{example}

The special Schreier extensions of \cite{martins2016baer} are extensions $\normalext{N}{k}{G}{e}{H}$ such that whenever $e(g) = e(g')$ there exists a \emph{unique} element $n \in N$ such that $k(n)g' = g$. Of course, every special Schreier extension is cosetal.
The following example from \cite{faul2020baer} exhibits a Leech extension, and hence a cosetal extension, which is not special Schreier. Thus, cosetal extensions constitute a nontrivial simultaneous generalization of Leech's extensions of groups by monoids and special Schreier extensions of monoids.

\begin{example}
 Let $\Z_\infty$ be the monoid obtained by adjoining an absorbing element $\infty$ to the integers under addition. Consider the extension $\normalext{\Z}{k}{\Z_\infty}{e}{\mathbbm{2}}$ in which $k(n) = n$ and $e(n) = \top$ for $n \in \Z$ and $e(\infty) = \bot$. Because $Z_\infty$ is commutative, it is clear that this is a Leech extension. However, it is not special Schreier as $n+\infty = \infty$ for all $n \in \Z$ and so uniqueness fails.
\end{example}

\section{Natural morphisms between the second cohomology groups}

Throughout this section, the kernel $N$ will be assumed to be an abelian group.
In the setting of special Schreier extensions there is an assignment of monoid actions $\phi$ of $H$ on $N$ to abelian groups $\mathcal{H}^2(H,N,\phi)$.
Since the set of actions $\Act(H,N)$ can be thought of as a discrete category, this assignment gives rise immediately to a functor $L_{H,N}\colon \Act(H,N) \to \Ab$.

For the theory of cosetal extensions we get an assignment of compatible pairs $(E,[\phi])$ to the abelian groups $\mathcal{H}^2(H,N,E,[\phi])$.
But in this case $\WAct(H,N)$ is a preorder and so functoriality of this assignment is a nontrivial property. In this section we demonstrate that it is indeed functorial and use it investigate the morphisms of cosetal extensions.

\subsection{The functor}\label{sec:the_functor_L}
Suppose that $(E,[\phi]) \leq (E', [\phi])$. We would like to find a relationship between $\mathcal{H}^2(H,N,E,[\phi])$ and $\mathcal{H}^2(H,N,E',[\phi])$.

\begin{proposition}\label{prop:grouphomo}
    If $(E,[\phi]) \leq (E', [\phi])$ then the map $\ell \colon \mathcal{H}^2(H,N,E,[\phi]) \to \mathcal{H}^2(H,N,E',[\phi])$ sending $[g]_E$ to $[g]_{E'}$ is a group homomorphism.
\end{proposition}

\begin{proof}
    Let us begin by showing that $\ell$ is well-defined.
    
    As $g$ is a factor set, we have $g(x,y)g(xy,z) \sim_E^{xyz} \phi(x,g(y,z))g(x,yz)$.
    Now since $E \subseteq E'$, we see that this identity is automatically satisfied with respect to $E'$ and so $g$ is also a factor set with respect to $E'$.
    
    Recall that two factor sets are equivalent in $\mathcal{H}^2(H,N,E,[\phi])$ if there exists an inner factor set $\delta t$ such that $g(h,h') \sim_E^{hh'} (\delta t + g')(h,h')$ for all $h,h' \in H$.
    But since $E \subseteq E'$, this is also satisfied with respect to $E'$. Thus, $g \equiv_{E'} g'$ and the mapping is well defined.
    
    It is clear from the definition that the mapping preserves the identity and addition.
\end{proof}

It is not hard to see that this makes the assignment of compatible pairs to the cohomology groups into a functor.
\begin{definition}
 We obtain a functor $L_{H,N}\colon \WAct(H,N) \to \Ab$ that sends $(E,[\phi])$ to $\mathcal{H}^2(H,N,E,[\phi])$ and yields the morphism $\ell\colon \mathcal{H}^2(H,N,E,[\phi]) \to \mathcal{H}^2(H,N,E',[\phi])$ described above for $(E,[\phi]) \le (E',[\phi])$.
\end{definition}

\subsection{Morphisms of cosetal extensions}
There is the immediate question of whether the extension corresponding to a cohomology class $[g]_E \in \mathcal{H}^2(H,N,E,[\phi])$ is in any way related to the extension corresponding to $\ell([g]) \in \mathcal{H}^2(H,N,E',[\phi])$.

\begin{proposition}\label{prop:canonical_map_between_extensions}
    Suppose $(E,[\phi]) \le (E',[\phi])$ and take $[g] \in \mathcal{H}^2(H,N,E,[\phi])$.
    Consider the extensions $\normalext{N}{k'}{N \rtimes_{E,\phi}^g H}{e'}{H}$ and $\normalext{N}{k}{N \rtimes_{E',\phi}^g H}{e}{H}$ associated $[g]$ and $\ell([g])$ respectively. Then the map $\lambda\colon N \rtimes_{E,\phi}^g H \to N \rtimes_{E',\phi}^{g} H$ defined by $\lambda([n],h) = ([n],h)$ is a morphism of extensions.
\end{proposition}
\begin{proof}
    This map is well-defined, since $E \subseteq E'$, and it is immediate that the identity $([1],1)$ is preserved. Preservation of multiplication follows from the following trivial calculation.
    \begin{align*}
        \lambda(([n],h)([n'],h')) &= \lambda([n\phi(h,n')g(h,h')],hh') \\
                            &= [n\phi(h,n')g(h,h')],hh') \\
                            &= ([n],h)([n'],h') \\
                            &= \lambda([n],h)\lambda([n'],h')
    \end{align*}
    We then clearly have $\lambda k = k'$ and that $e'\lambda = e$, which completes the proof.
\end{proof}

We can now ask a converse to the above question: if $f$ is a morphism between cosetal extensions, is there any way to relate their associated cohomology classes?

Consider the following morphism of extensions.
\begin{center}
   \begin{tikzpicture}[node distance=2.0cm, auto]
    \node (A) {$N$};
    \node (B) [right of=A] {$G_1$};
    \node (C) [right of=B] {$H$};
    \node (D) [below of=A] {$N$};
    \node (E) [right of=D] {$G_2$};
    \node (F) [right of=E] {$H$};
    \draw[normalTail->] (A) to node {$k_1$} (B);
    \draw[-normalHead] (B) to node {$e_1$} (C);
    \draw[normalTail->] (D) to node {$k_2$} (E);
    \draw[-normalHead] (E) to node {$e_2$} (F);
    \draw[->] (B) to node {$f$} (E);
    \draw[double equal sign distance] (A) to (D);
    \draw[double equal sign distance] (C) to (F);
   \end{tikzpicture}
\end{center}
To find the associated $H$-indexed equivalence relation for $\normalext{N}{k_1}{G_1}{e_1}{H}$, we take an arbitrary set-theoretic splitting $s$ of $e_1$ and define $n \sim^h n'$ if and only if $k_1(n)s(h) = k_1(n')s(h)$.

Notice that if $s$ is such a splitting, then $e_2fs = e_1s = 1$ and so $fs$ is a splitting of $e_2$. This observation allows us to prove many of the results that follow.

\begin{lemma}
    Let $f$ be a morphism of extensions as in the following diagram.
    \begin{center}
       \begin{tikzpicture}[node distance=2.0cm, auto]
        \node (A) {$N$};
        \node (B) [right of=A] {$G_1$};
        \node (C) [right of=B] {$H$};
        \node (D) [below of=A] {$N$};
        \node (E) [right of=D] {$G_2$};
        \node (F) [right of=E] {$H$};
        \draw[normalTail->] (A) to node {$k_1$} (B);
        \draw[-normalHead] (B) to node {$e_1$} (C);
        \draw[normalTail->] (D) to node {$k_2$} (E);
        \draw[-normalHead] (E) to node {$e_2$} (F);
        \draw[->] (B) to node {$f$} (E);
        \draw[double equal sign distance] (A) to (D);
        \draw[double equal sign distance] (C) to (F);
       \end{tikzpicture}
    \end{center}
    If $E_1$ and $E_2$ are the equivalence relations associated to $\normalext{N}{k_1}{G_1}{e_1}{H}$ and $\normalext{N}{k_2}{G_2}{e_2}{H}$ respectively, then $E_1 \subseteq E_2$.
\end{lemma}

\begin{proof}
    Let $s$ be a splitting of $e_1$ and suppose that $n \sim_{E_1}^h n'$. This means that $k_1(n)s(h) = k_1(n')s(h)$. Now recall that $fs$ splits $e_2$ and consider the following sequence of equalities.
    \begin{align*}
        k_2(n)fs(h) &= fk_1(n)fs(h) \\
                    &= f(k_1(n)s(h)) \\
                    &= f(k_1(n')s(h)) \\
                    &= k_2(n')fs(h).
    \end{align*}
    Thus, $n \sim_{E_2}^h n'$ and hence $E_1 \subseteq E_2$.
\end{proof}

Using similar ideas we can show that the same action can be extracted from both extensions.

\begin{lemma}
    Let $f$ be a morphism of extensions as in the following diagram.
    \begin{center}
       \begin{tikzpicture}[node distance=2.0cm, auto]
        \node (A) {$N$};
        \node (B) [right of=A] {$G_1$};
        \node (C) [right of=B] {$H$};
        \node (D) [below of=A] {$N$};
        \node (E) [right of=D] {$G_2$};
        \node (F) [right of=E] {$H$};
        \draw[normalTail->] (A) to node {$k_1$} (B);
        \draw[-normalHead] (B) to node {$e_1$} (C);
        \draw[normalTail->] (D) to node {$k_2$} (E);
        \draw[-normalHead] (E) to node {$e_2$} (F);
        \draw[->] (B) to node {$f$} (E);
        \draw[double equal sign distance] (A) to (D);
        \draw[double equal sign distance] (C) to (F);
       \end{tikzpicture}
    \end{center}
    There exists a map $\phi$ such that $\normalext{N}{k_1}{G_1}{e_1}{H}$ belongs to
    $\Cext(H,N,E_1,[\phi])$ and $\normalext{N}{k_2}{G_2}{e_2}{H}$ belongs to $\Cext(H,N,E_2,[\phi])$.
\end{lemma}
\begin{proof}
    Suppose that $s$ splits $e_1$ and that $(E_1,\phi)$ is the compatible pair associated to
    $\normalext{N}{k_1}{G_1}{e_1}{H}$. This means that $k_1\phi(h,n)s(h) = s(h)k_1(n)$. Applying $f$ to both sides yields $fk_1\phi(h,n)fs(h) = fs(h)fk_1(h)$ and hence $k_2\phi(h,n)fs(h) = fs(h)k_2(n)$. This shows that $\phi$ is compatible with $\normalext{N}{k_2}{G_2}{e_2}{H}$ as required.
\end{proof}

Thus, if $f$ is a morphism of extensions, we know that the extensions must correspond to cohomology classes in $\mathcal{H}^2(H,N,E_1,[\phi])$ and $\mathcal{H}^2(H,N,E_2,[\phi])$ with $(E_1,[\phi]) \le (E_2,[\phi])$. But in this situation, \cref{prop:grouphomo} supplies us with a map $\ell\colon \mathcal{H}^2(H,N,E_1,[\phi]) \to \mathcal{H}^2(H,N,E_2,[\phi])$. We now show that $\ell$ must map the one extension to the other.

\begin{lemma}
    Let $f$ be a morphism of extensions as in the following diagram.
    \begin{center}
       \begin{tikzpicture}[node distance=2.0cm, auto]
        \node (A) {$N$};
        \node (B) [right of=A] {$G_1$};
        \node (C) [right of=B] {$H$};
        \node (D) [below of=A] {$N$};
        \node (E) [right of=D] {$G_2$};
        \node (F) [right of=E] {$H$};
        \draw[normalTail->] (A) to node {$k_1$} (B);
        \draw[-normalHead] (B) to node {$e_1$} (C);
        \draw[normalTail->] (D) to node {$k_2$} (E);
        \draw[-normalHead] (E) to node {$e_2$} (F);
        \draw[->] (B) to node {$f$} (E);
        \draw[double equal sign distance] (A) to (D);
        \draw[double equal sign distance] (C) to (F);
       \end{tikzpicture}
    \end{center}
    If $\normalext{N}{k_1}{G_1}{e_1}{H}$ corresponds to $[g]_E \in \mathcal{H}^2(H,N,E_1,[\phi])$, then $\normalext{N}{k_2}{G_2}{e_2}{H}$ corresponds to $\ell([g]_E) \in \mathcal{H}^2(H,N,E_2,[\phi])$.
\end{lemma}

\begin{proof}
    Suppose that $s$ is a splitting of $e_1$ and then let $g_s$ be the associated factor set. We will show that $g_s$ is also a factor set associated to the splitting $fs$ of $e_2$.
    
    We know that $k_1g_s(h,h')s(hh') = s(h)s(h')$. Now simply apply $f$ to both sides to get $fk_1g_s(h,h')fs(hh') = fs(h)fs(h')$. This of course gives $k_2g_s(h,h')fs(hh') = fs(h)fs(h')$. This shows that $g_s$ is also an associated factor set of $fs$, which completes the proof.
\end{proof}

Together these results give the following theorem.
\begin{theorem}\label{prop:existence_of_morphisms_of_cosetal_extensions}
        Let $\normalext{N}{k_1}{G_1}{e_1}{H}$ be a cosetal extension corresponding to $[g_1]_{E_1} \in \mathcal{H}^2(H,N,E_1,[\phi_1])$ and let $\normalext{N}{k_2}{G_2}{e_2}{H}$ be a cosetal extension corresponding to $[g_2]_{E_2} \in \mathcal{H}^2(H,N,E_2,[\phi_2])$. Then 
        there exists a morphism $f$ of extensions as in the following diagram
    \begin{center}
       \begin{tikzpicture}[node distance=2.0cm, auto]
        \node (A) {$N$};
        \node (B) [right of=A] {$G_1$};
        \node (C) [right of=B] {$H$};
        \node (D) [below of=A] {$N$};
        \node (E) [right of=D] {$G_2$};
        \node (F) [right of=E] {$H$};
        \draw[normalTail->] (A) to node {$k_1$} (B);
        \draw[-normalHead] (B) to node {$e_1$} (C);
        \draw[normalTail->] (D) to node {$k_2$} (E);
        \draw[-normalHead] (E) to node {$e_2$} (F);
        \draw[->] (B) to node {$f$} (E);
        \draw[double equal sign distance] (A) to (D);
        \draw[double equal sign distance] (C) to (F);
       \end{tikzpicture}
    \end{center}
    if and only if $(E_1,[\phi_1]) \le (E_2,[\phi_2])$ and if  $[g_1]_{E_2} = [g_2]_{E_2}$. 
\end{theorem}

Notice that this implies that the only morphisms between extensions belonging to $\Cext(H,N,E,[\phi])$ are endomorphisms. In fact, we can show that they are all automorphisms.

\begin{proposition}\label{prop:automorphism}
    If $f$ is an endomorphism of the cosetal extension $\normalext{N}{k}{N \rtimes_{E,\phi}^g H}{e}{H}$, then $f$ is an automorphism.
\end{proposition}

\begin{proof}
    Let $s$ be a splitting of $e$.
    Since $f$ is an endomorphism of extensions, we have $fk = k$ and $ef = e$. Note that $f([n],h) = f(k(n)s(h)) = fk(n)fs(h) = k(n)fs(h)$. Because $ef = e$, we know $efs(h) = h$ and so $fs(h) = [f^*(h),h]$ for some $f^*(h) \in N$. Thus, $f([n],h) = ([f^*(h)n],h)$.
    
    It is now easily seen that $([n],h) \mapsto ([f^*(h)^{-1}n],h)$ provides an inverse function.
\end{proof}

In particular, we arrive at the following parameterised version of the short five lemma.

\begin{theorem}
        If $f$ is a morphism of extensions as in the following diagram
    \begin{center}
       \begin{tikzpicture}[node distance=2.0cm, auto]
        \node (A) {$N$};
        \node (B) [right of=A] {$G_1$};
        \node (C) [right of=B] {$H$};
        \node (D) [below of=A] {$N$};
        \node (E) [right of=D] {$G_2$};
        \node (F) [right of=E] {$H$};
        \draw[normalTail->] (A) to node {$k_1$} (B);
        \draw[-normalHead] (B) to node {$e_1$} (C);
        \draw[normalTail->] (D) to node {$k_2$} (E);
        \draw[-normalHead] (E) to node {$e_2$} (F);
        \draw[->] (B) to node {$f$} (E);
        \draw[double equal sign distance] (A) to (D);
        \draw[double equal sign distance] (C) to (F);
       \end{tikzpicture}
    \end{center}
    where $\normalext{N}{k_1}{G_1}{e_1}{H}$ and $\normalext{N}{k_2}{G_2}{e_2}{H}$ belong to $\Cext(H,N,E,[\phi])$, then $f$ is an isomorphism.
\end{theorem}

We can extend the ideas in \cref{prop:automorphism} to give a full characterisation of the automorphisms of extensions in terms of an obvious generalisation of \emph{crossed homomorphisms} from the theory of group extensions.

Let $\normalext{N}{k}{N \rtimes_{E,\phi}^g H}{e}{H}$ belong to $\Cext(H,N,E,[\phi])$. 
Notice that any automorphism $f$ must send $([n],h)$ to $([f^*(h)n],h)$, where $f^*(h)\colon H \to N$. We then ask for which functions $t^* \colon H \to N$ the map $t([n],h) = ([t^*(h)n],h)$ is a homomorphism.
For $t$ to preserve the unit we need that $t^*(1) = 1$. 
For it to preserve the multiplication we need for $t(([n],h)([n'],h')) = ([t^*(hh')n\phi(h,n')g(h,h')],hh')$
and $t([n],h)t([n'],h') = ([t^*(h)n\phi(h,t^*(h'))\phi(h,n')g(h,h')],hh')$
to be equal. This happens when
\[
    t^*(hh')n\phi(h,n')g(h,h') \sim^{hh'} t^*(h)n\phi(h,t^*(h'))\phi(h,n')g(h,h').
\]
Now multiplying on the left by $n\inv \phi(h,n')\inv g(h,h')\inv$ gives $t^*(hh') \sim^{hh'} t^*(h)\phi(h,t^*(h'))$, 
which is the natural weakening of a crossed homomorphism to our setting. 

\begin{definition}
    A function $t^* \colon H \to N$ is a crossed homomorphism relative to $(E,\phi)$ when 
    \[
        t^*(hh') \sim^{hh'} t^*(h)\phi(h,t^*(h')).
    \]
\end{definition}

We have shown above that crossed homomorphisms give rise to automorphisms of extensions and that every endomorphism arises in this way.
However, it is possible for two crossed homomorphisms to yield the same endomorphism. To remedy this we quotient the set of crossed homomorphisms with the equivalence relation defined by $t^*_1 \sim t^*_2 \iff t^*_1(h) \sim^h t^*_2(h)$. 

\begin{definition}
Let $\mathcal{Z}^1(H,N,E,[\phi])$ denote the set of equivalence classes of crossed homomorphisms. This inherits an abelian group structure from the pointwise multiplication of crossed homomorphisms.
\end{definition}

\begin{theorem}
    If $\Gamma$ is a cosetal extension $\normalext{N}{k}{G}{e}{H}$ in $\mathcal{H}^2(H,N,E,[\phi])$, then there is a bijection between $\End(\Gamma)$ and $\mathcal{Z}^1(H,N,E,[\phi])]$.
    Moreover, this an isomorphism of monoids. 
\end{theorem}

The same approach allows us to characterise arbitrary morphisms. Let $f$ be a morphism of extensions as in the following diagram.
\begin{center}
   \begin{tikzpicture}[node distance=2.0cm, auto]
    \node (A) {$N$};
    \node (B) [right=1.2cm of A] {$N \rtimes_{E,\phi}^g H$};
    \node (C) [right=1.2cm of B] {$H$};
    \node (D) [below of=A] {$N$};
    \node (E) [below of=B] {$N \rtimes_{E',\phi}^g H$};
    \node (F) [below of=C] {$H$};
    \node (G) [left=0.3cm of A] {$\Gamma\colon$};
    \node (H) [left=0.3cm of D] {$\Psi\colon$};
    \node (blank) [right=0.3cm of C] {};
    \draw[normalTail->] (A) to node {$k_1$} (B);
    \draw[-normalHead] (B) to node {$e_1$} (C);
    \draw[normalTail->] (D) to node {$k_2$} (E);
    \draw[-normalHead] (E) to node {$e_2$} (F);
    \draw[->] (B) to node {$f$} (E);
    \draw[double equal sign distance] (A) to (D);
    \draw[double equal sign distance] (C) to (F);
   \end{tikzpicture}
\end{center}

Then by a similar argument we have that $f([n],h) = ([t^*(h)n],h)$ where $t^* \colon H \to N$ is a crossed homomorphism relative to the admissible $H$-indexed equivalence relation $E'$. Thus, we arrive at the following result.

\begin{theorem}
Let $\Gamma$ and $\Psi$ be cosetal extensions as above. Then $\Hom(\Gamma,\Psi)$ is bijective to $\mathcal{Z}^1(H,N,E',\phi)$ under the assumptions of \cref{prop:existence_of_morphisms_of_cosetal_extensions} (and empty otherwise).
Here an element of $\mathcal{Z}^1(H,N,E',\phi)$ is sent to the composition of the corresponding automorphism on $\Psi$ and the map $\lambda\colon \Gamma \to \Psi$ as defined in \cref{prop:canonical_map_between_extensions}.
In particular, $\lambda$ corresponds to the the identity of $\mathcal{Z}^1(H,N,E',\phi)$.
\end{theorem}

\begin{remark}
 Composition with an automorphism in the codomain corresponds to translation in the group $\mathcal{Z}^1(H,N,E',\phi)$ by the associated crossed homomorphism.
 A crossed homomorphism from the domain naturally gives one in the codomain and
 composition with an automorphism in the \emph{domain} corresponds to translation by this corresponding crossed homomorphism.
\end{remark}

\section{Parameterisation by actions alone}

In \cite{faul2019characterisation} and \cite{faul2020baer} weakly Schreier and cosetal extensions are characterised in terms of admissible equivalence relations, compatible actions, and in the latter case, cohomology classes. Instead of starting with an admissible equivalence relation and then specifying an action compatible with it, 
it is also possible to start by specifying a candidate action and then choosing a compatible equivalence relation. By \cref{prop:existence_of_morphisms_of_cosetal_extensions} we know that different extensions admit morphisms between them only when they have the same action and cohomology class and so it can be helpful to consider all the compatible equivalence relations corresponding to a given candidate action.
To start we will not require $N$ to be an abelian group.

\begin{proposition}
    Let $N$ and $H$ be monoids and let $\alpha\colon H \times N \to N$ be an arbitrary function.
    The equivalence relations compatible with $\alpha$ are closed under inhabited pointwise intersections.
\end{proposition}
\begin{proof}
    All of the conditions for $\alpha$ to be compatible in \cref{def:compatible_action} and conditions (2) and (3) in \cref{def:admissible_indexed_equivalence_relation} are implications from one of the equivalence relations to another and so they are clearly closed under arbitrary pointwise intersections. Furthermore, condition (1) of \cref{def:admissible_indexed_equivalence_relation} is clearly downwards closed and hence closed under inhabited pointwise intersections.
\end{proof}

The equivalence relations compatible with a given map $\alpha\colon H \times N \to N$ are not closed under arbitrary meets in general, since it might happen that no such equivalence relation exists at all. For instance, it is shown in \cite{faul2019characterisation} that when $H$ is a group, only true actions are compatible with any equivalence relation.

\begin{definition}
 We say a function $\alpha\colon H \times N \to N$ is \emph{valid} if it is compatible with some $H$-indexed equivalence relation.
\end{definition}

Validity is, in fact, the \emph{only} obstruction to the compatible equivalence relations forming a complete lattice. To show this, we will prove that when the set is nonempty, it has a largest element.
We start by finding an indexed equivalence relation which contains every indexed equivalence relation compatible with $\alpha$.

\begin{proposition}
    Let $\alpha\colon H \times N \to N$ be compatible with an $H$-indexed equivalence relation $E$.
    Then whenever $n \sim_E^h n'$, we have $\forall x, y\in H.\ \left[ xhy = 1 \implies \alpha(x,n) = \alpha(x,n') \right]$. 
\end{proposition}
\begin{proof}
    Suppose that $n \sim_E^h n'$ and consider $x,y \in H$ such that $xhy = 1$. 
    Since $n \sim_E^h n'$, we have that $n \sim_E^{hy} n'$. Now applying the second compatibility condition, we obtain $\alpha(x,n) \sim_E^{xhy} \alpha(x,n')$. But $xhy = 1$ and so $\alpha(x,n) = \alpha(x,n')$ as required.
\end{proof}

This suggests the following definition.

\begin{definition}
    Given a function $\alpha\colon H \times N \to N$, we call the associated $H$-indexed equivalence relation $E_\alpha$ defined by
    \[
        n \sim_\alpha^h n' \iff \forall x, y\in H.\ \left[ xhy = 1 \implies \alpha(x,n) = \alpha(x,n')\right]
    \]
    the associated \emph{coarse equivalence relation}.
\end{definition}
Note that this is indeed an $H$-indexed equivalence relation.

\begin{theorem}
    If $\alpha$ is compatible with any $H$-indexed equivalence relation $E$, then it is compatible with the coarse equivalence relation $E_\alpha$.
\end{theorem}
\begin{proof}
    Let us begin by showing that $E_\alpha$ is admissible. 
    Suppose $n \sim^1_\alpha n'$ and observe that $x\cdot 1\cdot y = 1$ when $x = 1 = y$. Thus, $\alpha(1,n) = \alpha (1,n')$. Since $\alpha$ is compatible with $E$, we get that $n \sim^1_E n'$ and hence $n = n'$.
    
    Next we must show that if $n \sim_\alpha^h n'$ then $an \sim_\alpha^h an'$. Now suppose $n \sim_\alpha^h n'$ and $xhy = 1$. We know that $\alpha(x,n) = \alpha(x,n')$. Consider
    \begin{align*}
        \alpha(x,an)    &\sim_E^x \alpha(x,a)\alpha(x,n) \\
                        &\sim_E^x \alpha(x,a)\alpha(x,n') \\
                        &\sim_E^x \alpha(x,an').
    \end{align*}
    It follows that $\alpha(x,an) \sim_E^{xhy} \alpha(x,an')$. But $xhy = 1$ and so we find $\alpha(x,an) = \alpha(x,an')$.
    
    Finally, we must show that if $n \sim_\alpha^h n'$ then $n \sim^hh' n'$. Suppose $n \sim_\alpha^h n'$ and $xhh'y' = 1$. Then for $y = h'y'$ we have that $xhy = 1$. Thus, by assumption we get that $\alpha(x,n) = \alpha(x,n')$ and hence $E_\alpha$ is admissible.
    
    Now we show that $\alpha$ is compatible with $E_\alpha$. Notice that because $E \subseteq E_\alpha$ and $\alpha$ is compatible with $E$, we immediately have that conditions (3)--(6) hold. So we need only check conditions (1) and (2).
    
    To prove condition (1) we must show that if $n \sim_\alpha^h n'$ then $n\alpha(h,a) \sim_\alpha^h n'\alpha(h,a)$ for all $a \in N$. Suppose $xhy = 1$. By assumption we have $\alpha(x,n) = \alpha(x,n')$. We must show that $\alpha(x,n\alpha(h,a)) = \alpha(x, n'\alpha(h,a))$. 
    With respect to $E$ we have
    \begin{align*}
        \alpha(x,n\alpha(h,a))  &\sim_E^x \alpha(x,n) \cdot \alpha(x,\alpha(h,a)) \\
                                &\sim_E^x \alpha(x,n') \cdot \alpha(x,\alpha(h,a)) \\
                                &\sim_E^x \alpha(x,n'\alpha(h,a)).
    \end{align*}
    As above, it follows that $\alpha(x,n\alpha(h,a)) = \alpha(x,n'\alpha(h,a))$ since $x(hy)=1$.
    
    Finally, for condition (2) we must show that if $n \sim_\alpha^h n'$ then $\alpha(b,n) \sim_\alpha^{bh} \alpha(b,n')$ for all $b \in H$. Suppose $n \sim_\alpha^h n'$ and $xbhy = 1$. We must show $\alpha(x,\alpha(b,n)) = \alpha(x,\alpha(b,n'))$. Notice that by assumption $\alpha(xb,n) = \alpha(xb,n')$.
    Now consider
    \begin{align*}
        \alpha(x,\alpha(b,n))   &\sim_E^{xb} \alpha(xb,n) \\
                                &\sim_E^{xb} \alpha(xb,n') \\
                                &\sim_E^{xb} \alpha(x,\alpha(b,n')).
    \end{align*}
    Because $xb(hy) = 1$, we then find $\alpha(x,\alpha(b,n)) = \alpha(x,\alpha(b,n'))$, as required.
\end{proof}

\begin{corollary}\label{prop:compatible_quotients_complete_lattice}
 The equivalence relations compatible with any valid map $\alpha\colon H \times N \to N$ form a complete lattice.
\end{corollary}

The existence of the coarsest compatible equivalence relation was proved in \cite{faul2019characterisation} under the assumption that whenever $xh \in H$ is right invertible, so is $h$. In this case, every valid $\alpha$ has the same coarsest equivalence relation given by $n \sim^h n' \iff \forall y\in H.\ \left[ hy = 1 \implies n = n' \right]$. The following example gives a coarsest compatible equivalence relation, which does not satisfy these conditions.

\begin{example}
 Let $H$ be the bicyclic monoid $B = \langle p, q \mid pq = 1 \rangle$, $N$ the group $\Z^\omega$ of integer sequences under addition and $\alpha$ the true action $\alpha\colon H \times N \to N$ defined by \[\alpha(q, s)_n = \begin{cases} s_{n-1} & n > 0 \\ 0 & n = 0 \end{cases}\] and $\alpha(p, s)_n = s_{n+1}$.
 
 We can show that $x(q^a p^b)y = 1$ if and only if $x = p^{a+i}$ and $y = q^{b+i}$ for some $i \ge 0$
 and hence the coarse equivalence relation is given by $s \sim_\alpha^{q^a p^b} s' \iff \forall n \ge a.\ s_n = s'_n$.
 
 The resulting weak semidirect product can be expressed as the set of pairs of the form $(s, q^a p^b)$ where $s\colon \N \cap [a,\infty) \to \Z$ with unit $(0, 1)$ and multiplication given by
 \[
  (s, q^a p^b) \cdot (s', q^{a'} p^{b'}) = (n \mapsto s_n + s'_{n+b-a}, q^{a+a'-\min(b,a')}p^{b+b'-\min(b,a')}). 
 \]
 The kernel sends $s$ to $(s,1)$ and the cokernel sends $(s,x)$ to $x$.
\end{example}

By \cref{prop:compatible_quotients_complete_lattice} the equivalence relations compatible with a valid map $\alpha\colon H \times N \to N$ are closed under \emph{joins}. These correspond to meets in the order of the associated quotients. In \cite{faul2019artin} we showed that for Artin glueings of frames this meet operation can be interpreted as a kind of Baer sum.
Of course, there is a different notion of Baer sum for cosetal extensions with abelian kernel, in which case the equivalence relation is fixed beforehand.
We might attempt to gain some insight into the interaction of the equivalence relations and cohomology classes by combining them into a single structure. From here on we again assume $N$ is always an abelian group.

\begin{definition}
 Let $N$ be an abelian group and fix a valid action $\phi\colon H \times N \to N$.
 We define $\widetilde{H}^2_\phi(H,N)$ to be the set of pairs $(E,[g])$ where $E$ is an $H$-indexed equivalence relation compatible with $\alpha$ and $[g] \in \mathcal{H}^2(H, N, E, [\phi])$.
 
 We define an operation $(E_1, [g_1]) + (E_2, [g_2]) = (E_1 \vee E_2, \ell_1([g_1]) + \ell_2([g_2]))$
 on $\widetilde{H}^2_\phi(H,N)$ where the maps $\ell_{1,2}\colon \mathcal{H}^2(H,N,E_{1,2},[\phi]) \to \mathcal{H}^2(H,N,E_1 \vee E_2,[\phi])$ are defined as in \cref{prop:grouphomo}
 and we define a constant $0 = (\bot, 0) \in \widetilde{H}_\phi(H,N)$ where $\bot$ is the finest equivalence relation compatible with $\phi$.
\end{definition}

\begin{theorem}
 The algebra $\widetilde{\mathcal{H}}^2_\phi(H,N)$ is an inverse monoid where $(E, [-g])$ is the inverse of $(E,[g])$.
\end{theorem}
\begin{proof}
 The axioms are all routine calculations.
\end{proof}

\begin{remark}
The addition of this monoid appears to be related to the notion of Baer sum considered in \cite{stepp1971semigroup, fulp1971structure} in the special case of central extensions (which are precisely the cosetal extensions where the candidate action $\phi$ is trivial).
\end{remark}

This inverse monoid allows us to understand the relationship between all cosetal extensions of $H$ by $N$ with a given fixed valid action $\phi$. Of course, the choice of $\phi$ for a given extension is not unique, since it is only defined up to a quotient and so, unlike the cohomology groups $\mathcal{H}^2(H, N, E, [\phi])$, the objects $\widetilde{\mathcal{H}}^2_\phi(H,N)$ no longer partition the extensions,
though if the sum of two elements is taken in different monoids, the results will be compatible.

Furthermore, suppose we are given extensions with compatible pairs satisfying $(E_1, \phi_1), (E_2, \phi_2) \le (E_3, \phi_3)$, but where $\phi_1$ and $\phi_2$ are distinct with respect to $E_1$ and $E_2$.
In this case, we could analyse the relationship between the first and third extensions in $\widetilde{\mathcal{H}}^2_{\phi_1}(H,N)$ and between the second and the third in $\widetilde{\mathcal{H}}^2_{\phi_2}(H,N)$,
but there is no inverse monoid that would allow us to analyse all three.

One way to address this problem is to define the following \emph{category}.
\begin{definition}
 We define a category $\widetilde{\mathcal{H}}^2(H, N)$ whose objects are given by valid actions of $H$ on $N$ and where the morphisms from $\phi$ to $\phi'$ are given by pairs of the form $(E, [g])$ where $E$ is an $H$-indexed equivalence relation compatible with $\phi$ and $\phi'$ and with respect to which $\phi$ and $\phi'$ are equivalent and $[g] \in \mathcal{H}^2(H, N, E, [\phi])$. Composition is given by multiplication in $\widetilde{\mathcal{H}}^2_\phi(H, N)$.
\end{definition}

The inverse monoids $\widetilde{\mathcal{H}}^2_\phi(H,N)$ are then simply the endomorphism monoids in this category, but it also allows us to move from one of these inverse monoids to another as necessary. The category $\widetilde{\mathcal{H}}^2(H, N)$ is an example of what is known as an \emph{inverse category}.

\begin{definition}
 An inverse category is a category for which every morphism $f\colon X \to Y$ has a unique `inverse' $g\colon Y \to X$ such that $fgf = f$ and $gfg = g$.
\end{definition}

Any inverse category can be encoded as a certain kind of ordered groupoid (see \cite{deWolf2017}).
Applying this construction in our case, we obtain a category whose objects are compatible pairs $(E, \phi)$ and where the homset $\Hom((E, \phi), (E', \phi'))$ is given by $\mathcal{H}^2(H, N, E, [\phi])$ when $E = E'$ and $\phi \sim_E \phi'$ and is empty otherwise. Composition is given by multiplication in the obvious way. The order on objects is defined by $(E_1, \phi_1) \preceq (E_2, \phi_2)$ if and only if $E_1 \supseteq E_2$ and $\phi_1 = \phi_2$ and the order on morphisms is defined by $[g_1] \preceq [g_2]$ for $[g_{1,2}]\colon (E_{1,2}, \phi_{1,2}) \to (E_{1,2}, \phi'_{1,2})$
if and only if $E_1 \supseteq E_2$, $\phi_1 = \phi_2$, $\phi'_1 = \phi'_2$ and $[g_1] = \ell([g_2])$.

This order is somewhat too strict, since it distinguishes between compatible pairs with equivalent actions. This motivates the following definition.

\begin{definition}
 The ordered groupoid $\widehat{\mathcal{H}}^2(H, N)$ has compatible pairs $(E, [\phi])$ as objects with the reverse of the usual order. Similarly to the groupoid above, the morphisms from $(E,[\phi])$ to $(E',[\phi'])$ exist when $E = E'$ and $\phi$ is equivalent to $\phi'$ and are given by elements of $\mathcal{H}^2(H, N, E, [\phi])$. Composition is given in the obvious way and the morphisms can be ordered in an analogous way to above.
\end{definition}

Alternatively, we can apply the Grothendieck construction (see \cite[B.1.3.1]{elephant1} for details) to the functor $L_{H,N}\colon \WAct(H,N) \to \Ab$ from \cref{sec:the_functor_L} (composed with the inclusion $\Ab \hookrightarrow \Cat$) to obtain a category ${\int} L_{H,N}$ consisting of compatible pairs $(E, [\phi])$ as objects and with morphisms from $(E_1,[\phi])$ to $(E_2,[\phi])$ given by elements of $\mathcal{H}^2(H, N, E_2, [\phi])$ for $E_1 \subseteq E_2$.
The underlying groupoid of $\widehat{\mathcal{H}}^2(H, N)$ is then simply the \emph{core} (i.e.\ the groupoid of invertible morphisms) of ${\int} L_{H,N}$ equipped with the order induced by the reverse of the order on $\WAct(H,N)$.

The functoriality of these cohomology groups, inverse monoids and categories with respect to $H$ and $N$ will be explored in a later paper.

\bibliographystyle{abbrv}
\bibliography{bibliography}

\end{document}